\documentclass[12pt]{article}

\usepackage[utf8]{inputenc}
\usepackage[T1]{fontenc}
\usepackage{amsmath, amssymb, amsfonts}
\usepackage{authblk}
\usepackage{mathrsfs}

\usepackage{amsthm}
\usepackage{stmaryrd}

\usepackage{xurl}

\usepackage{tikz}
\usetikzlibrary{arrows.meta, positioning}

\title{\bf On the Category-Theoretic Independence of Meaning, Object, Name and Existence}
\author{\Large Takao Inou\'e\thanks{Corresponding author: inoue.takao@yamato-u.ac.jp \\
 (Personal) takaoapple@gmail.com (I prefer my personal mail)}}
\affil{\large Faculty of Informatics, \\Yamato University, Osaka, Japan}

\date{February 20, 2026}

\begin{document}
\maketitle

\begin{abstract}
We prove a category-theoretic independence theorem for four fundamental notions:
meaning, object, name, and existence.
Working in a Lawvere-style categorical semantics and in particular in toposes,
we show that these notions occupy distinct structural levels (object,
morphism, element, and internal logical level) and are not uniformly
recoverable from one another.

The key separation arises between internal existence and global naming.
Using a concrete example in the topos $\mathbf{Sh}(S^1)$—the sheaf of local
sections of a nontrivial covering—we exhibit an object that is internally
inhabited but admits no global element.

These results provide a precise structural basis for treating geometric
universes as foundational frameworks for information networks.
\end{abstract}

\newtheorem{definition}{Definition}[section]
\newtheorem{theorem}{Theorem}[section]
\newtheorem{lemma}{Lemma}[section]
\newtheorem{remark}{Remark}[section]

\newcommand{\keywords}[1]{%
  \par\bigskip
  \noindent\textbf{Keywords: }#1\par
}

\newcommand{\subjectclass}[1]{%
  \par\medskip
  \noindent\textbf{MSC2020: }#1\par
}


\subjectclass{Primary 18C05; Secondary 18B25, 03G30, 18F20}

\tableofcontents

\section{Introduction}

The notions of meaning, object, name, and existence play a foundational role
in logic and philosophy of language.
From Frege's distinction between \emph{Sinn} and \emph{Bedeutung}
to modern formal semantics, considerable effort has been devoted to clarifying
how linguistic expressions relate to structures interpreted as ``objects''.

Despite these developments, the structural relations among meaning, object,
name, and existence are often treated informally.
In particular, existence is frequently identified with the presence of an object,
while names are tacitly conflated with their referents.
Such identifications obscure the fact that these notions operate at
distinct categorical levels.

The purpose of this paper is to make this separation precise.
Working in a Lawvere-style categorical semantics and, in particular, in toposes,
we prove a category-theoretic independence theorem showing that meaning,
object, name, and existence are not uniformly recoverable from one another.

The central structural phenomenon appears in geometric universes:
internal existence, expressed via existential quantification in the internal logic,
does not coincide with global naming.
Using a concrete example in the topos $\mathbf{Sh}(S^1)$,
we exhibit an object that is internally inhabited but admits no global element.

These results provide a precise structural articulation of distinctions
that are often discussed philosophically, and they clarify the role of
geometric universes as frameworks in which existence is governed by
local coherence rather than global naming.

We conclude with remarks on possible directions for future research.

This work is a mathematical development of the conceptual analysis presented
in \cite{Inoue2026four}, where the distinctions among meaning, object, name,
and existence were investigated from a philosophical perspective within
Lawvere's categorical semantics.
Here we provide a precise category-theoretic formulation and an independence
theorem supporting those distinctions.

\section{Four Fundamental Notions (Formal Categorical Definitions)}

\subsection{Ambient categorical semantics}

Fix a category $\mathcal{E}$ with finite limits. For most of what follows,
it suffices to assume that $\mathcal{E}$ is a topos (so that $\Omega$ and
internal quantifiers exist), but the basic clauses below make sense already
in any finitely complete category.

We work with a typed language $\mathcal{L}$ (multi-sorted, simply typed, or
higher-order) whose \emph{syntactic/category of contexts and terms} we denote by
$\mathbf{Syn}(\mathcal{L})$. Objects of $\mathbf{Syn}(\mathcal{L})$ are contexts
$\Gamma$ (lists of typed variables), and arrows are (equivalence classes of)
terms-in-context.
A (Lawvere-style) \emph{categorical semantics} of $\mathcal{L}$ in $\mathcal{E}$
is given by an interpretation functor
\[
\llbracket-\rrbracket:\mathbf{Syn}(\mathcal{L})\longrightarrow \mathcal{E}
\]
preserving the relevant structure (at least finite limits; in the higher-order
case, usually cartesian closed structure and a subobject classifier).

For each type $A$ of $\mathcal{L}$, we write $\llbracket A\rrbracket$ for its
interpreted object in $\mathcal{E}$, and for each context $\Gamma$ we write
$\llbracket \Gamma\rrbracket$ for its interpreted object. If $t$ is a term of
type $A$ in context $\Gamma$, its interpretation is an arrow
\[
\llbracket t\rrbracket:\llbracket \Gamma\rrbracket \longrightarrow \llbracket A\rrbracket.
\]
If $\varphi$ is a formula in context $\Gamma$, its interpretation is a
subobject of $\llbracket\Gamma\rrbracket$, equivalently (in a topos) a
characteristic arrow
\[
\llbracket \varphi\rrbracket:\llbracket \Gamma\rrbracket \longrightarrow \Omega.
\]

\subsection{Formal definitions of the four notions}

\begin{definition}[Object]
An \emph{object} (relative to the semantics $\llbracket-\rrbracket$) is an object
of $\mathcal{E}$ that interprets a type of the language. Concretely, for a type
$A$ in $\mathcal{L}$, the \emph{object denoted by $A$} is $\llbracket A\rrbracket\in\mathrm{Ob}(\mathcal{E})$.
\end{definition}

\begin{definition}[Name]
Let $A$ be a type and $\mathcal{E}$ have a terminal object $1$.
A \emph{name} of an element of the object $\llbracket A\rrbracket$ is a
\emph{global element}
\[
n:1\longrightarrow \llbracket A\rrbracket.
\]
More generally, a \emph{generalized name} (or \emph{name in context} $X$) is an
arrow $n:X\to \llbracket A\rrbracket$. In the internal language, generalized
names correspond to terms with free variables, whereas global names correspond to
closed terms.
\end{definition}

\begin{definition}[Meaning]
The \emph{meaning} of an expression is its semantic value under $\llbracket-\rrbracket$.

\begin{enumerate}
\item If $t$ is a term of type $A$ in context $\Gamma$, the \emph{meaning of $t$}
is the morphism
\[
\llbracket t\rrbracket:\llbracket\Gamma\rrbracket\to \llbracket A\rrbracket.
\]
\item If $\varphi$ is a formula in context $\Gamma$ and $\mathcal{E}$ is a topos,
the \emph{meaning (truth-value) of $\varphi$} is the subobject
\[
\llbracket\varphi\rrbracket\hookrightarrow \llbracket\Gamma\rrbracket,
\]
equivalently its characteristic arrow $\llbracket\varphi\rrbracket:\llbracket\Gamma\rrbracket\to\Omega$.
\end{enumerate}

Thus, \emph{meaning} is fundamentally morphism-level data: terms denote arrows,
and formulas denote subobjects (or arrows into $\Omega$).
\end{definition}

\begin{definition}[Existence]
There are (at least) two distinct categorical notions of \emph{existence}.

\begin{enumerate}
\item \textbf{External (global) existence / inhabitedness.}
An object $\llbracket A\rrbracket$ is \emph{inhabited} (or \emph{has external existence})
if it has a global element, i.e.\ if there exists an arrow $1\to\llbracket A\rrbracket$.
Equivalently, the set $\mathcal{E}(1,\llbracket A\rrbracket)$ is nonempty.

\item \textbf{Internal (logical) existence.}
Assume $\mathcal{E}$ is a topos. For a context $\Gamma$ and type $A$, let
$\pi:\llbracket\Gamma\rrbracket\times \llbracket A\rrbracket\to \llbracket\Gamma\rrbracket$
be the projection. The \emph{existential quantifier along $\pi$} is the left adjoint
\[
\exists_\pi:\mathrm{Sub}(\llbracket\Gamma\rrbracket\times \llbracket A\rrbracket)\longrightarrow
\mathrm{Sub}(\llbracket\Gamma\rrbracket)
\]
to pullback $\pi^\ast$. If $\varphi$ is a formula in context $\Gamma,x\!:\!A$,
then the \emph{internal existence statement} $\exists x\!:\!A\,\varphi$ denotes the subobject
\[
\llbracket\exists x\!:\!A\,\varphi\rrbracket := \exists_\pi(\llbracket\varphi\rrbracket)
\hookrightarrow \llbracket\Gamma\rrbracket.
\]
\end{enumerate}

We stress that (i) inhabitedness is a statement about global elements, while
(ii) internal existence is a statement in the internal logic of $\mathcal{E}$.
They need not coincide in general.
\end{definition}

\begin{remark}[Separation of levels]
In this formalization:
\begin{itemize}
\item \emph{Object} lives at the level of objects of $\mathcal{E}$.
\item \emph{Name} lives at the level of (global or generalized) elements, i.e.\ arrows into an object.
\item \emph{Meaning} lives at the level of interpretation, i.e.\ arrows induced by $\llbracket-\rrbracket$.
\item \emph{Existence} lives either at the global level (inhabitedness) or at the internal logical level
(existential quantification), and these are categorically distinct notions.
\end{itemize}
\end{remark}

\section{Main Theorem: Category-Theoretic Independence}

\begin{lemma}[Internally inhabited sheaf without global name]\label{lem:S1}
Let $X=S^1$ be the circle, and let
\[
p:E\longrightarrow S^1
\]
be a nontrivial covering map (for example, the standard double covering
$p:S^1\to S^1$, $p(z)=z^2$).
Let $A$ be the sheaf of local sections of $p$, i.e.\ for each open set
$U\subseteq S^1$,
\[
A(U)=\{\, s:U\to E \mid p\circ s=\mathrm{id}_U \,\}.
\]
Then $A$ is internally inhabited in $\mathbf{Sh}(S^1)$, but has no global section.
\end{lemma}

\begin{proof}
We argue in two steps.

\smallskip\noindent
\emph{(Internal inhabitedness).}
By the local triviality of covering maps, every point of $S^1$ admits an open
neighborhood $U$ over which $p$ has a continuous local section.
Equivalently, the family of restriction maps
\[
A(U)\longrightarrow A(V)
\]
is locally nonempty.
Categorically, this means that the canonical morphism
\[
!_A:A\longrightarrow 1
\]
is an epimorphism in $\mathbf{Sh}(S^1)$.
Hence $\mathbf{Sh}(S^1)\models \exists a:A.\top$, so $A$ is internally inhabited.

\smallskip\noindent
\emph{(Absence of global sections).}
A global section of $A$ is exactly a continuous global section of the covering
$p:E\to S^1$.
Since $p$ is assumed to be nontrivial, no such global section exists.
Thus $\mathbf{Sh}(S^1)(1,A)=\varnothing$.

Combining the two parts, $A$ is internally inhabited but admits no global name.
\end{proof}

To make the term \emph{independence} precise, we adopt the following minimal
criterion.

\begin{definition}[Non-recoverability]
Fix a semantic environment $(\mathcal{E},\llbracket-\rrbracket)$ as in the previous section.
We say that a notion $\mathsf{N}$ is \emph{recoverable} from a collection of notions
$\mathsf{D}$ if, for every such environment, $\mathsf{N}$ is determined (up to canonical
isomorphism/equality in $\mathcal{E}$) by $\mathsf{D}$ \emph{uniformly}, i.e.\ by a construction
that is invariant under equivalence of semantic environments.
We say that $\mathsf{N}$ is \emph{independent} from $\mathsf{D}$ if it is \emph{not}
recoverable from $\mathsf{D}$.
\end{definition}

The next lemma records the standard separation between \emph{internal} and \emph{external}
notions of existence in a topos.

\begin{lemma}[Internal inhabitedness vs.\ global elements]\label{lem:inhabited}
Let $\mathcal{E}$ be a topos and $A\in\mathcal{E}$.
Write $!_A:A\to 1$ for the unique map to the terminal object.
Then:
\begin{enumerate}
\item $A$ has a \emph{global element} iff there exists an arrow $1\to A$.
\item $A$ is \emph{internally inhabited} (i.e.\ $\mathcal{E}\models \exists a:A.\top$)
iff $!_A:A\to 1$ is an epimorphism.
\end{enumerate}
In general, (2) does \emph{not} imply (1).
\end{lemma}

\begin{proof}
(1) is the definition of a global element.

(2) In the internal language of a topos, the statement ``$\exists a:A.\top$''
expresses that $A$ is inhabited \emph{locally}. The standard categorical
characterization is that local inhabitedness is equivalent to $!_A$ being epi.
(For completeness: $!_A$ epi means that the family of generalized elements of $A$
is jointly surjective in the sense of the internal logic, exactly matching the
validity of $\exists a:A.\top$.)

For non-implication, take $\mathcal{E}=\mathbf{Sh}(X)$ for a space $X$ admitting a
covering map $p:E\to X$ with no global section (e.g.\ a nontrivial covering of a
connected space). Let $A$ be the sheaf of local sections of $p$.
Then $!_A:A\to 1$ is epi (sections exist locally on a cover), hence $A$ is internally
inhabited; but there is no global section of $p$, hence no global element $1\to A$.
\end{proof}

We can now state a precise independence theorem. Intuitively: objects live at the
\emph{object level}, meanings at the \emph{morphism/subobject level}, names at the
\emph{element level}, and existence at the \emph{logical (epi/quantifier) level}.
These levels do not collapse in general.

\begin{theorem}[Category-theoretic independence]\label{thm:independence}
Let $(\mathcal{E},\llbracket-\rrbracket)$ range over semantic environments where
$\mathcal{E}$ is a topos and $\llbracket-\rrbracket$ is a Lawvere-style interpretation.
Then the four notions
\[
\textnormal{Meaning},\quad \textnormal{Object},\quad \textnormal{Name},\quad \textnormal{Existence}
\]
are pairwise and collectively independent in the following concrete sense:

\begin{enumerate}
\item \textbf{Meaning is not recoverable from Object, Name, and Existence.}
Even fixing the same interpreted objects and the same available names/existence facts,
distinct meanings (distinct arrows) can occur.

\item \textbf{Name is not recoverable from Meaning, Object, and Existence.}
Even fixing interpreted objects, meanings, and existence facts, the set of global names
may vary (and may even be empty).

\item \textbf{Existence is not recoverable from Meaning, Object, and Name.}
In particular, internal existence (inhabitedness) does not coincide with external existence
(global naming) in general.

\item \textbf{Object is not recoverable from Meaning, Name, and Existence.}
Even if one fixes all morphism-level data available in a fragment and the existence/name
facts visible there, non-isomorphic objects can remain indistinguishable.
\end{enumerate}
\end{theorem}

\begin{proof}
We give explicit counterexamples for each recoverability claim.

\smallskip\noindent
(1) \emph{Meaning not recoverable from the others.}
Work in $\mathcal{E}=\mathbf{Set}$. Let $A=B=\{0,1\}$.
Consider two distinct arrows $f,g:A\to B$ with $f=\mathrm{id}$ and $g$ the constant map $g(x)=0$.
The interpreted \emph{objects} $A,B$ are the same in both cases; the available \emph{names}
(global elements) of $A$ and $B$ are just their underlying elements, hence also the same;
and the crude \emph{existence} facts (inhabitedness, existence of global elements) are the same.
Yet the \emph{meaning} of the term (the arrow) differs: $f\neq g$.
Hence meaning is not recoverable from object-, name-, and existence-level data.

\smallskip\noindent
(2) \emph{Name not recoverable from the others.}
Work in a topos with few or no global points, e.g.\ $\mathcal{E}=\mathbf{Sh}(X)$ for a connected
space $X$.
There exist many non-isomorphic sheaves $A$ with no global sections, i.e.\ with
$\mathcal{E}(1,A)=\varnothing$, hence with \emph{no names} at all.
Fix any such $A$ and consider the same object-level and morphism-level fragment in which $A$
appears (e.g.\ its identity morphism, projections, etc.). All those \emph{meaning}-level arrows
are fixed, and the \emph{existence}-level statements may also be fixed (for instance, $A$ may be
internally inhabited or not, independently of having a global section).
Nevertheless, the availability of global names is not determined by the other data: one may
replace $A$ by $A+1$ (coproduct with the terminal object) which \emph{does} have a global element,
while preserving large parts of the same fragment of morphism-level behavior.
Hence name is not uniformly recoverable from meaning/object/existence.

\smallskip\noindent
(3) \emph{Existence not recoverable from the others.}
By Lemma~\ref{lem:S1}, in the topos $\mathbf{Sh}(S^1)$ there exists an object $A$
(the sheaf of local sections of a nontrivial covering of $S^1$) which is
internally inhabited, i.e.\ $\mathbf{Sh}(S^1)\models \exists a:A.\top$, but has no
global element $1\to A$.
Thus internal existence is not determined by object-, meaning-, or name-level
data.

\smallskip\noindent
(4) \emph{Object not recoverable from the others.}
In a general topos, the absence of global points causes many distinct objects to have identical
``name profiles'' (often empty) and to be indistinguishable by limited fragments of morphism-level
information.
Concretely, in $\mathbf{Sh}(X)$ for connected $X$, there are non-isomorphic sheaves $A\not\cong B$
with $\mathbf{Sh}(X)(1,A)=\mathbf{Sh}(X)(1,B)=\varnothing$.
If one restricts attention to a fragment where only existence/name facts and a small family of
morphisms are visible (as in typical semantic fragments), $A$ and $B$ cannot be recovered solely
from that data. Hence object is not uniformly recoverable from meaning/name/existence.
\end{proof}

\begin{remark}[How this supports geometric universes]
Theorem~\ref{thm:independence} shows that, in geometric universes (toposes), the semantic strata
(object / morphism / element / internal existence) are genuinely distinct. This provides a precise
structural basis for the claim that a Grothendieck-style geometric universe can serve as a foundation
for information networks: local truth and internal existence need not collapse to global naming.
\end{remark}

\section{An Informal Guide to the Independence Theorem}

The main theorem establishes that the four notions of meaning, object, name,
and existence are independent in a precise category-theoretic sense.
Since the formal statement and proof rely on categorical semantics and topos
theory, we provide here an informal guide intended for advanced students and
readers from philosophy or logic who may be less familiar with these tools.

\subsection{Four notions as four categorical levels}

The key idea is that the four notions live at different structural levels in a
categorical semantics:

\begin{itemize}
\item \emph{Objects} are interpreted as objects $A$ of a category (typically a
topos).
\item \emph{Meanings} of terms are interpreted as morphisms $f:A\to B$.
\item \emph{Names} correspond to global elements $1\to A$, i.e.\ morphisms from
the terminal object.
\item \emph{Existence} is expressed internally, via existential quantification
or, equivalently, by the epimorphicity of $A\to 1$.
\end{itemize}

Although these notions are related, the main theorem shows that none of them
can be reduced to, or reconstructed from, the others in a uniform way.

\subsection{A schematic picture}

The situation can be summarized by the following diagrammatic intuition:

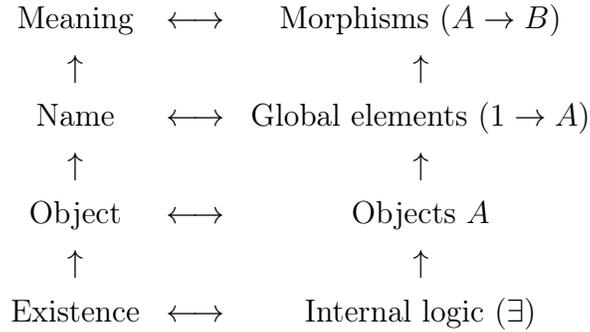
\[
\begin{array}{ccc}
\text{Meaning} & \longleftrightarrow & \text{Morphisms } (A \to B) \\[4pt]
\uparrow & & \uparrow \\[4pt]
\text{Name} & \longleftrightarrow & \text{Global elements } (1 \to A) \\[4pt]
\uparrow & & \uparrow \\[4pt]
\text{Object} & \longleftrightarrow & \text{Objects } A \\[4pt]
\uparrow & & \uparrow \\[4pt]
\text{Existence} & \longleftrightarrow & \text{Internal logic } (\exists)
\end{array}
\]

The vertical arrows should \emph{not} be read as definitional reductions.
The main theorem asserts precisely that, in general, there is no way to move
upwards in this diagram so as to reconstruct one notion from those below it.

\subsection{Why existence and naming separate}

The most instructive case is the distinction between existence and naming.
In ordinary set-theoretic semantics, to exist often means to be named by an
element. In a geometric universe, however, existence is a local notion.

In the topos $\mathbf{Sh}(S^1)$, the sheaf of local sections of a nontrivial
covering is internally inhabited: locally, sections exist.
Nevertheless, no global section exists, and hence no name is available.
This single example already shows that existence cannot be identified with
naming, nor reduced to it.

\subsection{Reading the formal proof}

Each part of the proof of the main theorem corresponds to one failure of
reconstruction in the above schematic picture:
\begin{itemize}
\item distinct meanings with the same objects and names (in $\mathbf{Set}$);
\item absence or presence of names without changing meaning or existence;
\item internal existence without global names (in $\mathbf{Sh}(S^1)$);
\item non-isomorphic objects indistinguishable by limited semantic data.
\end{itemize}

The formal proof makes these failures precise using standard categorical
constructions. Conceptually, however, they reflect the fact that geometric
semantics separates local coherence from global identification.

\begin{figure}[t]
\centering
\begin{tikzpicture}[
  box/.style={draw, rounded corners, align=center, minimum width=3.6cm, minimum height=1cm},
  arrow/.style={->, thick},
  dashedarrow/.style={->, thick, dashed}
]

\node[box] (meaning) {Meaning\\{\small morphisms $A\to B$}};
\node[box, below=of meaning] (name) {Name\\{\small global elements $1\to A$}};
\node[box, below=of name] (object) {Object\\{\small objects $A$}};
\node[box, below=of object] (existence) {Existence\\{\small internal logic $\exists$}};

\draw[arrow] (meaning) -- (name);
\draw[arrow] (name) -- (object);
\draw[arrow] (object) -- (existence);

\draw[dashedarrow] (existence.east) -- ++(2.2,0) |- (object.east);
\draw[dashedarrow] (existence.east) -- ++(2.8,0) |- (name.east);
\draw[dashedarrow] (existence.east) -- ++(3.4,0) |- (meaning.east);

\draw[dashedarrow] (object.west) -- ++(-2.2,0) |- (meaning.west);
\draw[dashedarrow] (name.west) -- ++(-2.8,0) |- (meaning.west);

\end{tikzpicture}

\caption{Four categorical levels and their independence.
Meaning, name, object, and existence live at distinct structural levels in a
categorical semantics. Solid arrows indicate structural relations, while dashed
arrows indicate \emph{non-recoverability} established by the independence
theorem. In particular, internal existence does not determine global naming.}
\label{fig:independence}
\end{figure}
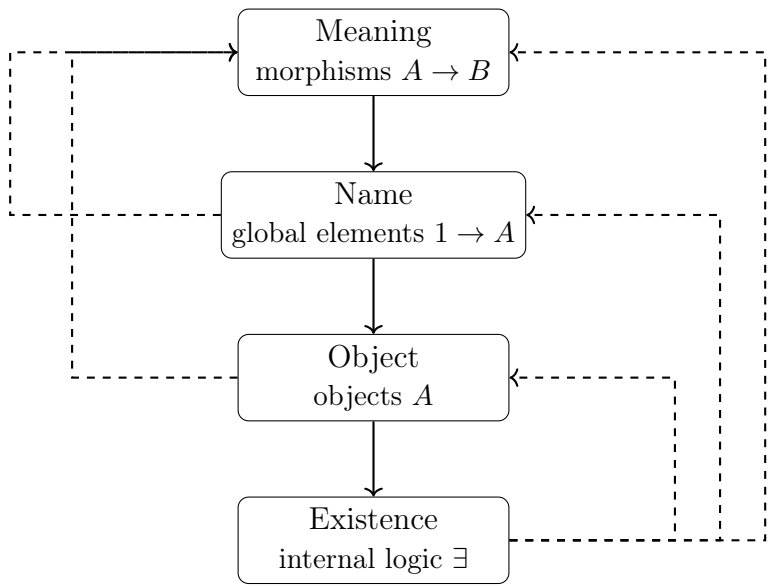

As illustrated in Figure~\ref{fig:independence}, the four notions form a
stratified structure: while they are related, the independence theorem shows
that no upward reconstruction is available in general.

\section{Conclusion}

In this paper, we established a category-theoretic independence theorem for
four fundamental notions: meaning, object, name, and existence.
Within a Lawvere-style categorical semantics and, in particular, in toposes,
we showed that these notions occupy distinct structural levels and are not
uniformly recoverable from one another.

A central separation concerns existence and naming.
Using a concrete example in the topos $\mathbf{Sh}(S^1)$, we exhibited an object
that is internally inhabited but admits no global element.
This demonstrates that existence, as expressed in the internal logic of a
geometric universe, cannot in general be reduced to objecthood or to the
availability of names.

These results provide a precise categorical foundation for distinctions that
have often been treated only conceptually, and they clarify how geometric
universes support a semantics in which existence is governed by local
coherence rather than global identification.

\section{The Final Remark}

The present work should be understood as a mathematical development of earlier
philosophical investigations.
In particular, the conceptual analysis of meaning, object, name, and existence
within Lawvere's categorical semantics presented in
\cite{Inoue2026four} is here strengthened by an explicit category-theoretic
independence theorem.

Lawvere's framework provides a powerful and unifying basis for categorical
semantics.
At the same time, the results of this paper show that further structural
distinctions—especially between internal existence and global naming—become
visible when one works explicitly in geometric universes.
In this sense, the present work may be regarded as a case study demonstrating
how such distinctions can be made mathematically precise within, and beyond,
Lawvere's original setting.

These observations suggest several directions for further development, 
including the treatment of names as modal objects, a direction that will
be explored in future work.

\begin{remark}
The independence theorem proved in this paper can be viewed as a
category-theoretic strengthening of the philosophical distinctions proposed in
\cite{Inoue2026four}, providing a structural foundation for them in geometric
universes.
\end{remark}

\begin{remark}
The distinctions established here also contribute to the geometric
interpretation of information networks developed in \cite{Inoue2026Network},
where local coherence, rather than global naming, plays a central structural
role.
\end{remark}

$$ $$

\bigskip 

\noindent Takao Inou\'{e}

\noindent Faculty of Informatics

\noindent Yamato University

\noindent Katayama-cho 2-5-1, Suita, Osaka, 564-0082, Japan

\noindent inoue.takao@yamato-u.ac.jp
 
\noindent (Personal) takaoapple@gmail.com (I prefer my personal mail)

\end{document}